\documentclass[12pt, oneside, reqno]{amsart}
\usepackage{amsmath}
\usepackage{amssymb}
\usepackage{centernot}
\usepackage{color}
\usepackage[dvipsnames,x11names,svgnames]{xcolor}
\usepackage[colorlinks=true,linkcolor=NavyBlue,citecolor=DarkGreen, urlcolor=blue]{hyperref}
\usepackage{amsthm}
\usepackage{amscd}
\usepackage{geometry}
\usepackage{mathrsfs}
\usepackage{dsfont}
\usepackage{cite}
\usepackage{mathtools}

\newcommand{\1}{\mathds{1}}

\newcommand{\Z}{\mathbb{Z}}

\newcommand{\N}{\mathbb{N}}

\newcommand{\F}{\mathbb{F}}

\newtheorem{theorem}{Theorem}[section]
\newtheorem{lemma}[theorem]{Lemma}

\theoremstyle{definition}

\theoremstyle{remark}

\newcommand{\act}{\curvearrowright}
\newcommand{\cO}{\mathcal{O}}
\newcommand{\fp}{\mathfrak{p}}
\begin{document}
 
\title[Footnote to a theorem of Phagan]{Footnote to a theorem of Phagan on actions of a cyclic group}
\author{Anurag Sahay}
\email{\href{mailto:anuragsahay@purdue.edu}{anuragsahay@purdue.edu},\href{mailto:anurag.sahay@duke.edu}{anurag.sahay@duke.edu}}
\urladdr{\href{https://www.math.purdue.edu/~sahay5/}{https://www.math.purdue.edu/$\sim$sahay5/}}

\begin{abstract}

We provide a simplified proof of a recent theorem of Phagan \cite{phagan} relating the number of orbits of a given length of an action $G \act S$ of a cyclic group with the number of orbits of the induced action $H \act S$ of a subgroup $H \subseteq G$. This combinatorial fact has applications to notions of arithmetic similarity of number fields, as investigated by Phagan (op. cit.).

\end{abstract}

\maketitle

\section{Introduction}

Our goal is to present a conceptual (or, at any rate, a simplified) proof of the following theorem of Phagan \cite[Theorem 2.3]{phagan}, which is a fact about group actions of a cyclic group:
\begin{theorem}[Phagan] \label{thm:phagan}
For $n \in \N$, let $G = \Z/n\Z$ be the finite cyclic group of order $n$, acting on a finite set $S$. For $m \in \N$, define $A(m;S)$ to be the number of orbits of the action $G \act S$ which have cardinality precisely $m$. Further, define
\[ M(m;S) = \#\big(H \backslash S\big), \]
the number of orbits under the induced action $H \act S$ of $H = m\Z/n\Z$, the subgroup generated by $m$ as an element of $G$.

With these definitions, we have, for $m \in \N$, 
\begin{equation} \label{maineqn} A(m;S) = \sum_{m \mid d \mid n} \frac{\mu(d/m)}{\phi(d)} \sum_{k \mid d} \mu(k) M(d/k;S), \end{equation}
where $\mu$ and $\phi$ are the usual arithmetic functions of M\"obius and Euler respectively. The first sum here runs over divisors $d$ of $n$ which are themselves multiples of $m$. In particular, the sum is empty if $m \nmid n$.
\end{theorem}
As indicated before, this elementary theorem was proven recently by Phagan \cite{phagan}, though we have presented his result with some cosmetic changes:
\begin{itemize} 
\item Most notably, we write the outer sum in this way instead of writing it equivalently as a sum over divisors of $n/m$. 
\item We choose to keep the dependence on $S$ explicit in the notation, since this will be helpful in our presentation.
\item We allow $m$ to vary freely in $\N$, instead of restricting to divisors of $n$.
\item We use slightly different notation for the cyclic groups, the arithmetic functions, et cetera.
\end{itemize}

Using this theorem, whose proof is elementary, Phagan is able to ``effectivize'' results about various notions of ``similarity'' of number fields. As a sample of his results, for a finite extension $K$ of $\mathbb{Q}$ and a rational prime $p$, consider 
\begin{equation} \label{idealfactorization} p\cO_K = \fp_1^{e_1} \fp_2^{e_2} \cdots \fp_g^{e_g}, \end{equation}
the unique factorization of the ideal $p\cO_K$ into prime ideals of the ring of integers, $\cO_K$. Let us denote the number of distinct prime ideals in \eqref{idealfactorization} by $g(p,K)$. It follows from a result of Stuart and Perlis \cite[Main Theorem]{stuartperlis} that if $L$ and $K$ are two number fields with the property that for every rational prime $p$ that does not ramify in either $L$ or $K$ one has the equality $g(p,K) = g(p,L)$, then in fact their Dedekind zeta functions coincide,
\[ \zeta_K(s) = \zeta_L(s). \]
Furthermore, this means that every prime $p$ has the same \emph{splitting type} in both $L$ and $K$, this being the multiset $\{f_1,\cdots,f_g\}$ of inertial degrees $f_j = \dim_{\F_p} \big(\cO_K/\fp_j\big)$ for the prime ideals appearing in \eqref{idealfactorization} (and similarly with $L$). Stuart and Perlis's result is ineffective, in that given the sequence $g(p,K)$ for unramified $p$, it is not clear how to compute the splitting types or the zeta function $\zeta_K(s)$. Using Theorem~\ref{thm:phagan}, Phagan gives formulae \cite[Corollary 2.4]{phagan} and \cite[Theorem 3.2]{phagan} determining the splitting types of each prime from the sequence $g(p,K)$ for unramified primes from which one could reconstruct the zeta function if desired. The reader may consult Phagan's article \cite{phagan} for other applications. 

We present our alternative proof of Theorem~\ref{thm:phagan} in \S\ref{sec:proof}. In \S\ref{sec:analysis}, we briefly reflect on the proof and speculate on generalizations to arbitrary groups.

\subsection*{Acknowledgments} The author was partially supported by Trevor Wooley's start-up funding at Purdue University and by the AMS-Simons Travel Grant at the time when this note was conceived. I am grateful to Shaver Phagan for several illuminating discussions about his work and related topics. I am also thankful to Trevor Wooley for his encouragement and to an anonymous referee for pointing out the reference \cite{crelleplesken} to me.

\section{Proof of Theorem~\ref{thm:phagan}} \label{sec:proof}
The starting point of our proof is similar to Phagan's. Define 
\begin{equation} \label{Ndef} N(m;S) = \frac{1}{\phi(m)} \sum_{k \mid m} \mu(k) M(m/k;S). \end{equation}
 
We will prove the identity
\begin{equation} \label{mainidentity} N(m;S) = \sum_{m \mid d} A(d;S). \end{equation}
It is worth remarking at this point that the sequences $N$ and $A$ are supported on divisors of $n$. For $A(m;S)$ this follows from the fact any orbit of $G$ is isomorphic as a $G$-set to a quotient of $G$. For $N(m;S)$, this will follow from \eqref{mainidentity}, since if $A(d;S) \neq 0$ for some multiple $d$ of $m$, then $d \mid n$ and hence $m \mid n$. Thus, in \eqref{mainidentity}, we could change the condition under the sum to $m \mid d \mid n$. 

From \eqref{mainidentity}, the desired identity \eqref{maineqn} follows by dual M\"obius inversion (see, for example, \cite[Exercise 1.5.16]{murty}). Note that there are no convergence issues here, since all the sequences involved are finitely supported. Alternately -- and perhaps, more simply -- this may be interpreted as M\"obius inversion on the poset of integers which lie between $m$ and $n$ in the divisibility lattice. In this way, we obtain
\[ A(m;S) = \sum_{m \mid d} \mu(d/m) N(d;S) \]
which yields \eqref{maineqn} by recalling that $N(d;S)$ vanishes when $d \nmid n$ and then substituting \eqref{Ndef}, the definition of $N$.

To prove \eqref{mainidentity}, we diverge from Phagan's work. The key simplifying lemma is the following. 
\begin{lemma} \label{lem:additivity}
Let $X \in \{ A, M, N\}$. Then, 
\[ X(m;S_1 \oplus S_2) = X(m;S_1) + X(m;S_2), \]
where $\oplus$ is the direct sum of $G$-sets. Further, if $S \cong T$ are two isomorphic $G$-sets, then 
\[ X(m;S) = X(m;T). \]
\end{lemma}
Before proving the lemma, let us recall the definition of the direct sum $S_1 \oplus S_2$ of two $G$-sets. The underlying set is the disjoint union $S_1 \sqcup S_2$, and the action $G \act S_1 \sqcup S_2$ is simply the action that restricts to the actions $G \act S_1$ and $G \act S_2$ on elements of $S_1$ and $S_2$ respectively. In particular, we see that the first part of the lemma amounts to saying that the sequences $A$, $M$, and $N$ are additive in the second variable while the second part to saying that they are invariant under $G$-set isomorphism.

\begin{proof}[Proof of Lemma~\ref{lem:additivity}]
The set of orbits of $G \act S_1 \oplus S_2$ is precisely the union of the orbits of $G \act S_1$ and $G \act S_2$ respectively. Clearly, then, the number of orbits of a given size in $S_1\oplus S_2$ is the sum of the number of such orbits in $S_1$ and $S_2$ respectively. This establishes
\[ A(m;S_1\oplus S_2) = A(m;S_1) + A(m;S_2). \]
A similar argument establishes additivity of $M$, simply by replacing $G$ by $H = m\Z/n\Z$ and by counting all orbits instead of orbits of a particular size. It is easy to see invariance under $G$-set isomorphism of $A$ and $M$, so we omit the proof. Finally, the additivity and isomorphism-invariance of $N$ follows by definition, since it is defined as a linear combination where each term is additive and isomorphism-invariant in the $G$-set.
\end{proof}

Thus, both sides of the identity \eqref{mainidentity} are clearly additive in $S$. If it holds for $S=S_j$ with $j \in \{1,2\}$, it must consequently hold for $S = S_1 \oplus S_2$. Therefore, to establish the general case, it suffices to establish it in the case where $S$ cannot be decomposed as a direct sum $S_1 \oplus S_2$ with $|S_1|, |S_2| < |S|$. A moment's reflection should convince the reader that $S$ lacks such a decomposition precisely when $G \act S$ is a transitive group action. 

Every transitive $G$-set is isomorphic to the set of cosets $G/N$ of some subgroup $N$ of $G$, with $G$ acting on the left by multiplication. By the isomorphism-invariance of both sides of \eqref{mainidentity}, we can thus further assume that $S$ is a quotient of $G$, and hence $S = \Z/\ell \Z$ for some divisor $\ell \mid n$. The action may be described transparently as follows: for an element $u \in G = \Z/n\Z$ and an element $v \in S = \Z/\ell\Z$, the action of $u$ on $v$ is obtained by first reducing $u$ modulo $\ell$ and then adding it to $v$.

For concision, let us write $A_\ell(m) = A(m; \Z/\ell\Z)$ and similarly for $M$ and $N$. Since the action is transitive, there is only one orbit which is of size $\ell$, whence
\[ A_\ell(m) = \1\big[m = \ell\big], \]
where here and throughout $\1[P]$ is the indicator function of a predicate $P$. From this it follows that the right hand side of \eqref{mainidentity} is
\[ \sum_{m\mid d} A_\ell(d) = \sum_{m\mid d} \1\big[d = \ell\big] = \1\big[m \mid \ell\big]. \] 

On the other hand, 
\[ M_\ell(m) = (\ell,m), \]
the greatest common divisor of $\ell$ and $m$. This requires a small calculation: every orbit of $m\Z/n\Z \act \Z/\ell\Z$ has size equal to the order of $m$ in $\Z/\ell\Z$, which is $\ell/(\ell,m)$, thus the number of orbits must be $(\ell,m)$. Now, in the definition of $N$, \eqref{Ndef}, we replace $k$ with the complementary divisor $m/k$, obtaining
\[ N_\ell(m) = \frac{1}{\phi(m)} \sum_{k \mid m} \mu(m/k) M_\ell(k) = \frac{1}{\phi(m)} \sum_{k \mid m} \mu(m/k) (\ell,k). \]

Thus, to prove \eqref{mainidentity}, we must show the identity
\begin{equation} \label{exercise} \frac{1}{\phi(m)} \sum_{k \mid m} \mu(m/k) (\ell,k) = \1\big[m \mid \ell\big] \end{equation}
for $m,\ell \in \N$. This is now a standard exercise in M\"obius inversion; one way to show this is to insert the identity
\[ q = \sum_{d \mid q} \phi(d) \]
applied to $q = (\ell,k)$ in the left hand side of \eqref{exercise}, reverse the order of summation, and then follow one's nose. This completes the proof.
\section{Reflections and speculations} \label{sec:analysis}
Let us briefly compare our proof to Phagan's. Phagan's proof is inductive, involving an induction essentially on $\omega(m)$, the number of prime divisors of $m$ counted without repetition and it uses the orbit-stabilizer theorem in the base case. Phagan does not change the $G$-set in the course of his proof. Our proof is also inductive, albeit on the cardinality of the $G$-set, $|S|$ instead of on $\omega(m)$. Our presentation hides this in the remark reducing to the case where $G \act S$ is a transitive action. The group theoretic input in our proof is not orbit-stabilizer but instead a factorization theorem of $G$-sets under direct sum. 

The simplifying Lemma~\ref{lem:additivity} has the flavour of the Burnside ring of a group (introduced in \cite{burnside, solomon}, see \cite{bouc} for a modern account). The author suspects there is more to be said here, but does not feel well-equipped to speculate on how this connection may be developed. A referee for this paper indicated that a generalization of Theorem~\ref{thm:phagan} to arbitrary finite groups may be thus obtained and suggested that such a result is already contained in \cite{crelleplesken}. The author was unable to make this precise for himself, especially with regard to deducing \eqref{maineqn} from \cite{crelleplesken}.

\bibliography{footnotebib}
\bibliographystyle{plain}

\end{document}